\documentclass[a4paper,12pt]{article}
\usepackage[russian,ukrainian,english]{babel}
\usepackage[cp1251]{inputenc}

\usepackage{amscd}
\usepackage{amsthm}
\usepackage{amssymb}
\usepackage{amsmath,amsthm,amsfonts}
\usepackage{hhline}
\usepackage{varioref}

\newtheorem{exam}{Example}[section]
\newtheorem{theorem}{Theorem}[section]
\newtheorem{lemma}{Lemma}[section]

\newtheorem{ozn}{Definition}[section]
\newtheorem{remk}{Remark }[section]
\newtheorem{nas}{Corollary}[section]
\newtheorem{prop}{Proposition}[section]
\theoremstyle{definition}

\DeclareMathOperator{\rank}{rank}
\DeclareMathOperator{\sign}{sign}

\newcommand{\cc}{\mathbb{C}}
\newcommand{\rr}{\mathbb{R}}
\newcommand{\nn}{\mathbb{N}}

\newcommand{\ff}{\mathbb{F}}

\title{Topological conjugacy classes of affine maps}

\author {Budnytska Tetiana}
\date{}
\begin{document}

\maketitle

\indent

\begin{abstract}
A map $f: \ff^n \to \ff^n$ over a field $\ff$ is called affine if it is of the form $f(x)=Ax+b$, where the matrix $A \in \ff^{\,n\times n}$ is called the linear part of affine map and $b \in \ff^n$.

The affine maps over $\ff=\rr$ or $\cc$ are investigated.
We prove that affine maps having fixed points are topologically conjugate if and only if their linear parts are topologically conjugate.
If affine maps have no fixed points and $n=1$ or $2$, then they are topologically conjugate if and only if their linear parts are either both singular or both non-singular. Thus we obtain classification up to topological conjugacy of affine maps from $\ff^n$ to $\ff^n$, where $\ff=\rr$ or $\cc$, $n\leq 2$.
\end{abstract}

\indent

\section{Introduction.}

We consider the problem of classification of affine maps up to topological conjugacy.

Throughout this paper we denote by $\ff^n$ the vector space $\rr^n$ or $\cc^n$, $n\geq 1$.

The maps $f$, $g: \ff^n \rightarrow \ff^n$ are said to be \emph{topologically conjugate} (written $f\stackrel{t}{\sim}g$), if there exists a homeomorphism $h: \ff^n \rightarrow \ff^n$ such that $g=h \circ  f\circ h^{-1}$.

The matrices $A$, $C \in \ff^{\,n\times n}$ are said to be \emph{topologically conjugate}, if the corresponding linear maps $f$, $g: \ff^n \rightarrow \ff^n$, $f(x)=Ax$ and $g(x)=Cx$ are topologically conjugate.

An element $x \in \ff^n$ is called a \emph{fixed point} of $f: \ff^n \rightarrow \ff^n$ if $f(x)=x$.

A map $f: \ff^n \rightarrow \ff^n$ of the form $f(x)=Ax+b$, $A \in \ff^{\,n\times n}$ and $b \in \ff^n$ is called an \emph{affine map} and the matrix $A$ is called the \emph{linear part} of $f$.

\pagebreak

The main result of the paper is the following theorem:
\begin{theorem}

Let $f$, $g: \ff^n \rightarrow \ff^n$, $f(x)=Ax+b$, $g(x)=Cx+d$  be affine maps and \:$\ff^n=\rr^n$ or $\cc^n$, $n\geq 1$.

\textbf{1.} If each of \,$f$ and $g$ has at least one fixed point, then $f \stackrel{t}{\sim} g$ if and only if their linear parts are topologically conjugate.

\textbf{2.} If $f$ and $g$ have no fixed points and $n=1$ or $2$, then $f\stackrel{t}{\sim}g$ if and only if \,$\det A$ and $\det C$ are either simultaneously equal to 0 or simultaneously different from 0.
\end{theorem}

\indent


\section{Topological classification of linear maps from $\rr^n$ to $\rr^n$, $n\geq1$.}

We will use the classification of linear maps up to topological conjugacy for the classification of affine maps, therefore we recall some known results about linear maps.

Let $f: \rr^n \rightarrow \rr^n$, $f(x)=Ax$ be a linear map. Then the \textbf{real canonical form} (written ${\cal R}_A$ or RCF) of the matrix $A$ (see~\cite{Palis}) may be written in the form:
$$
\left (
\begin{tabular} {cccc}
$A_+$ & 0     & 0          & 0     \\
0     & $A_-$ & 0          & 0     \\
0     & 0     & $A_\infty$ & 0     \\
0     & 0     & 0          & $A_0$ \\
\end{tabular} \right ),
$$
where the matrices $A_\alpha$, $\alpha$ =~$+$,~$-$,~$\infty$,~$0$, satisfy conditions listed in the following table:

$$
\begin{tabular}{|p{3.1cm}|p{4.0cm}|} \hline
Matrix   & The eigenvalues $\lambda$ of this matrix\\ \hline
$A_+ $              & $0<|\lambda |<1$    \\ \hline
$A_-$               & $|\lambda |>1$      \\ \hline
$A_\infty$          & $\lambda=0$         \\ \hline
$A_0$               & $|\lambda|=1$       \\ \hline
\end{tabular}
$$

For a linear map $f: \cc^n \rightarrow \cc^n$, $f(z)=Az$ by using the \textbf{Jordan canonical form} (written $J_A$) of the matrix $A$, the matrices $A_\alpha$, $\alpha$ =~$+$,~$-$, $\infty$,~$0$ are defined similarly.

\indent

The maps $f$, $g: \rr^n \rightarrow \rr^n$ are said to be \emph{linearly conjugate} (written $f\stackrel{\ell}{\sim}g$), if there exists a linear bijection $h: \rr^n \rightarrow \rr^n$ such that $g=h\circ f\circ h^{-1}$.

Clearly, linear conjugacy implies topological conjugacy.

A map $f: \rr^n \rightarrow \rr^n$ is called \emph{periodic} if there is $k \in \nn$ such that $f^k=id_{\rr^n}$. The smallest possible $k$ is called the \emph{period} of map $f$.

Kuiper and Robbin~\cite{Kuip-Robb, Robb} gave the topological classification of those linear maps from $\rr^n$ to $\rr^n$, whose matrices have no eigenvalues which are the roots of unity.

\begin{theorem}\cite{Kuip-Robb,Robb} \label{klas-matr-liniu-vidob}
Let $f$, $g: \rr^n \rightarrow \rr^n$, $f(x)=Ax$, $g(x)=Cx$ be linear maps, whose matrices have no eigenvalues which are the roots of unity.
$$
\begin{array}{lllllll}

\mbox{Then}\:\, f\stackrel{t}{\sim}g \:\, \mbox{if and only if}&
$$
\left\{
\begin{array}{ll}
\rank(A_+)=\rank(C_+), \: \sign(\,\det(A_+)\,)=\sign(\,\det(C_+)\,),\\
\rank(A_-)=\rank(C_-), \: \sign(\,\det(A_-)\,)=\sign(\,\det(C_-)\,), \\
A_\infty = C_\infty, \\
A_0 = C_0.
\end{array} \right.
$$

&\qquad &\qquad &\qquad
\end{array}
$$
\end{theorem}

\begin{remk}
The equalities \,$A_\infty = C_\infty$ and $A_0 = C_0$ hold up to the order of diagonal blocks of the matrices.
\end{remk}

Kuiper and Robbin~\cite{Kuip-Robb} showed that the general problem reduced to the case when $A$ and $C$ were periodic matrices. They conjectured that for all periodic linear maps topological conjugacy is equivalent to linear conjugacy, and they proved this when periods of maps are equal to $s$~=~1, 2, 3, 4 or 6.

Later, other conditions for which topological conjugacy is equivalent to linear conjugacy were found.
The problem of the topological classification of matrices with eigenvalues which are the roots of unity was partially solved by Kuiper and Robbin~\cite{Kuip-Robb, Robb}, Cappell and Shaneson~\cite{Capp-conexamp, Capp-2th-nas-n<=6, Capp-big-n<6, Capp-differ, Capp-Contemp-n<6}, Hsiang and Pardon~\cite{Pardon}, Madsen and Rothenberg~\cite{Madsen}, Schultz~\cite{Schultz}.

Among these papers we will consider the works of Cappell and Shaneson. They found counterexample to the conjecture of Kuiper and Robbin~\cite{Capp-conexamp} and they proved that for periodic linear maps from $\rr^n$ to $\rr^n$, $n\leq 5$ topological conjugacy is equivalent to linear conjugacy~\cite{Capp-2th-nas-n<=6,Capp-big-n<6,Capp-Contemp-n<6}.

Since \textbf{necessary and sufficient condition for linear conjugacy of two linear maps is equivalent to the equality of RCF's of their matrices}, it follows that two periodic linear maps from $\rr^n$ to $\rr^n$, $n\leq 5$, \,$f(x)=Ax$ and \,$g(x)=Cx$ are topologically conjugate if and only if $A_0 = C_0$.

The following proposition is a conclusion of the results of Kuiper and Robbin, Cappell and Shaneson.

\begin{prop}\label{klas-vsix-liniu-vidob-5-5}
Let $f$, $g: \rr^n \rightarrow \rr^n$, $n\leq 5$, $f(x)=Ax$, $g(x)=Cx$ be linear maps.
$$
\begin{array}{ll}

\mbox{Then} \:\, f\stackrel{t}{\sim}g \:\: \mbox{if and only if} &
$$
\left\{
\begin{array}{ll}
\rank(A_+)=\rank(C_+), \: \sign(\,\det(A_+)\,)=\sign(\,\det(C_+)\,),\\
\rank(A_-)=\rank(C_-), \: \sign(\,\det(A_-)\,)=\sign(\,\det(C_-)\,), \\
A_\infty = C_\infty, \\
A_0 = C_0.
\end{array} \right.
$$

\qquad \qquad \qquad
\end{array}
$$

\end{prop}

\indent

\textbf{Classification up to topological conjugacy of all periodic linear maps has not been solved completely yet, and so it is still unsolved for all linear maps.
}

\indent


\section{Topological classification of affine maps from $\rr^n$ to $\rr^n$, $n\geq1$.}

\subsection{Classification of affine maps from $\rr^n$ to $\rr^n$, $n\geq1$, such that each of them has at least one fixed point.}

We begin with classification up to topological conjugacy of affine maps from $\rr^n$ to $\rr^n$, $n\geq1$. It is based on the existence of a fixed point of the affine map. For the large class of affine maps this problem was completely solved.

\indent

The following theorem is one of the most important results in the topological classification of affine maps.
\begin{theorem}\cite{st-1}\label{pro-neryx-tochk}
Let $f$, $g: \rr^n\rightarrow \rr^n$, $f(x)=Ax+b$ and $g(x)=Ax$.

Then $f\stackrel{t}{\sim}g$ if and only if there is $q \in\rr^n$ such that $f(q)=q$.
\end{theorem}

\begin{proof}
{\it{Necessity.}}
$f(x)=Ax+b\stackrel{t}{\sim}g(x)=Ax$, therefore the maps $f$ and $g$ have the same number of fixed points.

Since $g(0)=0$, it follows that $g(x)=Ax$ has at least one fixed point, therefore $f(x)=Ax+b$ also has at least one fixed point.

{\it{Sufficiency.}}
By assumption there is $q \in \rr^n$ such that $f(q)=q$.

Let $h(x)=x+q$, then it is easy to see that $f=h\circ g\circ h^{-1}$, i.e. $f\stackrel{t}{\sim}g$.
\end{proof}

\begin{theorem}\label{klas-affin-vidob}
Let $f$, $g: \rr^n \rightarrow \rr^n$, $f(x)=Ax+b$, $g(x)=Cx+d$ be affine maps such that:
\begin{enumerate}
  \item[1)] each of $f$ and $g$ has at least one fixed point;
  \item[2)] the matrices $A$ and $C$ have no eigenvalues which are the roots of unity.
\end{enumerate}
$$
\begin{array}{ll}

Then \:\, f\stackrel{t}{\sim}g \:\:\mbox{if and only if} &
$$
\left\{
\begin{array}{ll}
\rank(A_+)=\rank(C_+), \: \sign(\,\det(A_+)\,)=\sign(\,\det(C_+)\,),\\
\rank(A_-)=\rank(C_-), \: \sign(\,\det(A_-)\,)=\sign(\,\det(C_-)\,),\\
A_\infty  =  C_\infty , \\
A_0  =  C_0 .
\end{array} \right.
$$
\qquad \qquad \qquad
\end{array}
$$
\end{theorem}

\begin{proof}
By assumption each of $f$ and $g$ has at least one fixed point.

By Theorem~\ref{pro-neryx-tochk}: \:$f(x)=Ax+b \stackrel{t}{\sim} r(x)=Ax$,

\qquad \qquad \qquad \qquad         $g(x)=Cx+d \stackrel{t}{\sim} s(x)=Cx$.

Employing Theorem~\ref{klas-matr-liniu-vidob} to the maps $r(x)=Ax$ and $s(x)=Cx$ we obtain the result.
\end{proof}

\textbf{Consequently, necessary and sufficient conditions of topological conjugacy of affine maps such that each of them has at least one fixed point coincide with the necessary and sufficient conditions of topological conjugacy of their linear parts.}

\indent

\indent
\subsection{Classification of affine maps from $\rr$ to $\rr$.}

In this paper we completely solve the problem of topological classification of affine maps from $\rr^n$ to $\rr^n$, $n=1$, $2$.

\indent

At first consider the one dimensional case.

For the linear maps the following proposition is true.

\begin{prop}~\cite{Kuip-Robb}\label{klas-lin-vidob-R}
Let $f:\rr\rightarrow \rr$, $f(x)=ax$, $a \in \rr$ be a linear map.

There are seven topological conjugacy classes of linear maps from $\rr$ to $\rr$. Three classes are given by the values $a=0$, $1$, $-1$ and the remaining classes are represented by the four remaining open intervals between the values $0$,~$1$ and~$-1$ on $\rr$.

That is, if $f(x)=ax$, $g(x)=cx$, $a$, $c \in \rr$, then
$f\stackrel{t}{\sim}g$ if and only if \,$a$ and $c$ either simultaneously belong to the same component of $\rr \setminus \left\{0,1,-1\right\}$ or simultaneously equal to one of the values $0$, $1$ and $-1$.
\end{prop}

For the affine maps we have a similar proposition.

\begin{prop}\cite{st-1}\label{klas-afin-vidob-R}
Let $f:\rr\rightarrow \rr$, $f(x)=ax+b$, $a$, $b \in \rr$ be an affine map.

There are eight topological conjugacy classes of affine maps from $\rr$ to $\rr$. Two classes are given by the values $a=0$, $-1$, another two ones by the pairs of the values $a=1$, $b=0$ and $a=1$, $b\neq0$ and the remaining classes are represented by the four remaining open intervals between the values $0$,~$1$ and~$-1$ on $\rr$.

That is, if $f(x)=ax+b$, $g(x)=cx+d$, $a$, $b$, $c$, $d \in\rr$, then $f\stackrel{t}{\sim}g$ if and only if \,$a$ and $c$ either simultaneously belong to the same component of $\rr \setminus \left\{0,1,-1\right\}$ or simultaneously equal to one of the values \,$0$, $1$ and $-1$, if \,$a=c=1$, then $b$ and $d$ either simultaneously equal to 0 or simultaneously different from 0.
\end{prop}

\begin{proof}
The topologically conjugate maps have the same topological properties. We list those topological properties which divide the set of maps $f(x)=ax+b$, $a, b \in \rr$ into eight classes:
\begin{enumerate}
    \item[1)] $f(x)=b\equiv const$ if and only if $a=0$.
    \item[2)] $f$ is the identity if and only if $a=1$ and $b=0$.
    \item[3)] $f$ does not have any fixed point if and only if $a=1$ and $b\neq0$.
    \item[4)] $f^2$ is the identity and $f$ is not the identity if and only if $a=-1$.
    \item[5)] $f$ preserves orientation if and only if $a>0$.
    \item[6)] $\lim \limits_{n\rightarrow\infty} f^n(x)= \frac{b}{1-a}\in \rr$ for all $x \in\rr$ if and only if $|a|<1$.
\end{enumerate}

Properties 5 and 6 divide $\rr$ into four intervals and properties $1 - 4$ generate points $0$, $1$ and $-1$ between these intervals. As the number $a=1$ with $b=0$ and $b\neq0$ generates two different classes of topologically conjugate maps (properties 2 and 3\,), it follows that we get as a result the eight classes of topologically conjugate affine maps.

If $f(x)=ax+b$, $a\neq1$, $b \in \rr$ and $g(x)=cx+d$, $c\neq1$, $d \in \rr$, where $a$ and $c$ either simultaneously belong to the same component of $\rr \setminus \left\{0,1,-1\right\}$ or simultaneously equal to $0$ or $-1$, then $g=h\circ f\circ h^{-1}$, where
$$
\begin{array}{ll}

h(x)=
$$
\left\{
\begin{array}{lllll}
\: \frac{d}{1-c} \:,\qquad \qquad \qquad \qquad \qquad \qquad \qquad \quad  x=\frac{b}{1-a},       \\
(x-\frac{b}{1-a})|x-\frac{b}{1-a}|^{l-1}+\frac{d}{1-c}\:,\:|a|^l=|c|,\:\: x \neq \frac{b}{1-a}.
\end{array} \right.
$$
\end{array}
$$

\indent

The maps $f(x)=x+b$, $b\neq0$ and $g(x)=x+d$, $d\neq0$ are topologically conjugate, because $g=h\circ f\circ h^{-1}$, where $h(x)=\frac{d}{b}~x+m$, $m \in \rr$.
\end{proof}

\begin{remk}
If in Proposition~\ref{klas-afin-vidob-R} we put the map $f(x)=ax+b$, where $a\in \rr$,\, $b=0$, then we obtain Proposition~\ref{klas-lin-vidob-R}, that is quite natural, because the set of all linear maps is the subset of all affine maps.
\end{remk}

Notice that in the monograph~\cite{Pelyx} authors found solutions of functional equation $h \circ f= g\circ h$, where $h$ is a unknown map from $\rr$ to $\rr$ and $f$ and $g$ present a wide class of maps from $\rr$ to $\rr$.

\indent

\subsection{Classification of affine maps from $\rr^2$ to $\rr^2$.}

It is known that topologically conjugate maps have the same number of fixed points. Therefore the classification of affine maps from $\rr^2$ to $\rr^2$ up to topological conjugacy is divided into two cases:
\begin{enumerate}
  \item[1)] topological classification of affine maps such that each of them has at least one fixed point;
  \item[2)] topological classification of affine maps, which don't have fixed points.
\end{enumerate}

\indent

\textbf{1) Topological classification of affine maps from $\rr^2$ to $\rr^2$ such that each of them has at least one fixed point.}\qquad

By Theorem~\ref{pro-neryx-tochk} and Proposition~\ref{klas-vsix-liniu-vidob-5-5} we have:

\begin{prop}\cite{st-2}\label{klas-affin-vidob-2-2}
Let $f$, $g: \rr^2 \rightarrow \rr^2$, $f(x)=Ax+b$, $g(x)=Cx+d$ be affine maps such that each of them has at least one fixed point.
$$
\begin{array}{lllll}

\mbox{Then} \:\, f\stackrel{t}{\sim}g \:\: \mbox{if and only if} &
$$
\left\{
\begin{array}{lllll}
\rank(A_+)=\rank(C_+), \: \sign(\,\det(A_+)\,)=\sign(\,\det(C_+)\,),\\
\rank(A_-)=\rank(C_-), \: \sign(\,\det(A_-)\,)=\sign(\,\det(C_-)\,),\\
A_\infty = C_\infty, \\
A_0 = C_0.
\end{array} \right.
$$

&\qquad &\qquad &\qquad
\end{array}
$$

\end{prop}

\indent

\textbf{2) Topological classification of affine maps from $\rr^2$ to $\rr^2$, which don't have fixed points.}

At first we formulate some necessary results.

\begin{lemma}\cite{st-2}\label{isnye-lambda_(A)=1}
Let $f: \rr^2 \rightarrow \rr^2$, $f(x)=Ax+b$ be an affine map.

If the map $f$ does not have any fixed point, then one of eigenvalues of the matrix $A$ is equal to 1.
\end{lemma}

\begin{proof}
The map $f(x)=Ax+b$ does not have any fixed point if there is no $x \in \rr^2$ such that $f(x)=x$, thus the system $(A-E)x=-b$\, does not have any solution on $\rr^2$.

Assume that the matrix $A$ does not have eigenvalue which is equal to 1, that is, $\det(A-E)\neq 0$.
Then the system $(A-E)x=-b$ has the solution $\tilde{x}=-(A-E)^{-1}\, b$, therefore the map $f$ has the fixed point, that contradicts to the condition of the lemma.
\end{proof}

\begin{remk}
Henceforth if an affine map $f: \rr^2 \rightarrow \rr^2$, $f(x)=Ax+b$ does not have any fixed point and $\lambda^{(1)}_A$, $\lambda^{(2)}_A$ are eigenvalues of the matrix $A$, then we suppose that $\lambda^{(1)}_A=1$.
\end{remk}


Propositions~\ref{ne-isnye-neryx.-toch-1,1-(a_{12})}, \ref{ne-isnye-neryx.-toch-1,d}  are easy and ideologically identical, because they give necessary and sufficient conditions when concrete affine maps don't have fixed points. These results will be used in the proofs of Theorems~\ref{detA-ne-0+topol. spryaj. z syvom} and \ref{detA=0, detC=0}.

\begin{prop}\cite{st-2}\label{ne-isnye-neryx.-toch-1,1-(a_{12})}
$$
\begin{array}{cccc}

\mbox{Let}\:\: f: \rr^2 \rightarrow \rr^2,\:  f(x)=
$$
\left(\begin{array}{cc}
1  & 1 \\
0  & 1 \\
\end{array}\right)
$$
$$
\left(\begin{array}{cc}
x_1 \\
x_2 \\
\end{array}\right)
$$
+
$$
\left(\begin{array}{cc}
b_1\\
b_2 \\
\end{array}\right),
$$
~ b_1, b_2\in \rr.
\qquad & \qquad & \qquad  \quad
\end{array}
$$

The map $f$ does not have any fixed point if and only if \,$b_2 \neq 0$.
\end{prop}

\begin{prop}\cite{st-2}\label{ne-isnye-neryx.-toch-1,d}
$$
\begin{array}{cccc}

\mbox {Let}\:\: f: \rr^2 \rightarrow \rr^2,\:  f(x)=
$$
\left(\begin{array}{cccc}
1  & 0 \\
0  & \alpha \\
\end{array}\right)$$
 $$
\left(\begin{array}{cccc}
x_1 \\
x_2 \\
\end{array}\right)
$$
+
$$
\left(\begin{array}{cccc}
\delta_1\\
\delta_2 \\
\end{array}\right),
$$
\:\, \alpha\in \rr\backslash \{1\},\: \delta_1, \delta_2\in \rr. \qquad
\end{array}
$$

The map $f$ does not have any fixed point if and only if \,$\delta_1 \neq 0$.
\end{prop}

\begin{prop}\cite{st-2}\label{topol. spryaj. zsyviv}
Let $f$, $g: \rr^2 \rightarrow \rr^2$, $f(x)=x+b$, $g(x)=x+d$, $b$,~$d \in \rr^2$ be affine maps.

Then $f\stackrel{t}{\sim}g$ if and only if \,$b$ and $d$ are either simultaneously equal to 0 or simultaneously different from 0.
\end{prop}

\begin{proof}
If $b\neq0$ and $d\neq0$, then $f\stackrel{t}{\sim}g$, because there is a homeomorphism $\psi: \rr^2 \rightarrow \rr^2$,
$$
\begin{array}{cc}

\psi(x)=Bx,\:\: \mbox {where the matrix} \: B \: \mbox {has properties}: \:
$$
\left\{ \begin{array}{ll}
 \det B\neq 0, \\
 Bb=d
\end{array} \right.
$$
\:\mbox {and} \:\: g=\psi \circ f\circ \psi^{-1}.
\qquad \qquad \qquad
\end{array}
$$

If $b\neq0$ and $d=0$ (or $b=0$ and $d\neq0$), then one of the maps is the identity and another is not the identity, therefore $f$ and $g$ are not topologically conjugate (because they have the different numbers of fixed points).
\end{proof}

\indent

As topologically conjugate maps are either simultaneously bijective or simultaneously non-bijective, it follows that classification up to topological conjugacy of affine maps from $\rr^2$ to $\rr^2$, which don't have fixed points is divided into two cases:
\begin{enumerate}
\item[a)] topological classification of bijective affine maps, which don't have fixed points;
\item[b)] topological classification of non-bijective affine maps, which don't have fixed points.
\end{enumerate}

\indent

\textbf{a) Topological classification of bijective affine maps, which don't have fixed points.}

The following theorem is the basic result of the case a) because it asserts that a bijective affine map, which doesn't have any fixed point and $g(x)=x+d$, $d\in\rr^2 \backslash \{0\}$ are topologically conjugate.

\begin{theorem}\cite{st-2}\label{detA-ne-0+topol. spryaj. z syvom}
Let $f$, $g: \rr^2 \rightarrow \rr^2$, $f(x)=Ax+b$, $g(x)=x+d$ be affine maps such that:
\begin{enumerate}
  \item[1)] $f$ and $g$ have no fixed points;
  \item[2)] $\det A\neq0$.
\end{enumerate}
Then $f\stackrel{t}{\sim}g$.
\end{theorem}

\begin{proof}
The idea of proof: for maps $f$ and $g$ we will construct a homeomorphism $h: \rr^2 \rightarrow \rr^2$ such that $g=h \circ f\circ h^{-1}$.
$$
\begin{array}{lll}

\mbox {If the map} \:f \:\mbox {satisfies conditions of this theorem, namely:}\:
$$
\left\{
\begin{array}{ll}
\nexists \: q \in \rr^2: f(q)=q,\\
\det A\neq 0,
\end{array} \right.
$$
&\qquad \qquad \qquad
\end{array}
$$
$$
\begin{array}{ll}
\mbox{then by Lemma~\ref{isnye-lambda_(A)=1}:}\:
$$
\left\{
\begin{array}{l}
\lambda^{(1)}_{A}=1,\\
\lambda^{(2)}_{A}\neq 0.
\end{array} \right.
$$
\qquad \qquad \qquad \qquad \qquad\qquad \qquad \qquad \qquad \qquad \qquad  \qquad \qquad \qquad \qquad \qquad
\end{array}
$$

We will divide the proof of this theorem into two possible cases:
\begin{enumerate}
  \item[1)] $\lambda^{(1)}_{A}=1$ and $\lambda^{(2)}_{A}=1$;
  \item[2)] $\lambda^{(1)}_{A}=1$ and $\lambda^{(2)}_{A}=\alpha$, where $\alpha \in \rr \setminus\{0,1\}$
\end{enumerate}
and construct the corresponding homeomorphisms in each of cases.

\indent

\indent

\textbf{Case 1: $\lambda^{(1)}_A=1$ and $\lambda^{(2)}_A=1$.}

Notice that if matrix $A$ is a unity matrix, then this theorem follows from Proposition~\ref{topol. spryaj. zsyviv}. Therefore we prove it only for a non-unity matrix.
$$
\begin{array}{cccc}

\mbox{We will show that a map}\:\, f(x)=Ax+b \:\,\mbox{such that}\:\,
$$
\left\{
\begin{array}{lllll}
\nexists \: q\in \rr^2: f(q)=q,\\
\lambda^{(1)}_{A}=1=\lambda^{(2)}_{A},\\
A\neq E
\end{array} \right.
$$
\mbox {and}
&\qquad &\qquad &\qquad
\end{array}
$$
$g(x)=x+d$, $d \in \rr^2\backslash\{0\}$ are topologically conjugate.

Construction of the corresponding homeomorphism is divided into four steps.
$$
\begin{array}{cccc}

\mbox{\textbf{1.} Let} \:\:  f(x)=Ax+b \:\: \mbox {be such that}\:\:
$$
\left\{
\begin{array}{lll}
\nexists \: q\in \rr^2: f(q)=q,\\
\lambda^{(1)}_{A}=1=\lambda^{(2)}_{A},\\
A\neq E
\end{array} \right.
$$
\:\: \mbox{and homeomorphism}
\qquad
\end{array}
$$
$$
\begin{array}{cc}

\:\,\phi_1: \rr^2 \rightarrow \rr^2,\:\,  \phi_1(x)=Sx, \: \mbox {where matrix} \: S: \:
$$
\left\{ \begin{array}{ll}
\det S\neq 0, \\
SAS^{-1}={\cal R}_A, \\
\end{array}\right.
$$
\:\: \mbox {then}\:\: f_1=\phi_1 \circ f\circ \phi_1^{-1},
&\qquad \qquad \qquad \qquad
\end{array}
$$
$$
\begin{array}{cccc}

\mbox{where}\:\:f_1(x)={\cal R}_Ax+\delta=
$$
\left(\begin{array}{cc}
1  & 1 \\
0  & 1
\end{array}\right)$$
$$
\left(\begin{array}{cc}
x_1 \\
x_2 \\
\end{array}\right)
$$
+
$$
\left(\begin{array}{cc}
\delta_1 \\
\delta_2 \\
\end{array}\right),
$$
\:\: \delta=Sb, \:\, \delta_1 \in \rr, \, \delta_2 \in \rr \backslash\{0\}.
\qquad \qquad \qquad \qquad
\end{array}
$$

Conditions $\delta_1 \in \rr$, $\delta_2 \in \rr \backslash\{0\}$ follows from Proposition~\ref{ne-isnye-neryx.-toch-1,1-(a_{12})}, because the map $f_1(x)={\cal R}_Ax+\delta$ does not have any fixed point. Indeed, $f_1 \stackrel{t}{\sim} f$ and $f$ does not have any fixed point by the assumption of the theorem.

$$
\begin{array}{cccc}

\mbox {\textbf{2.} The maps} \:\:  f_1(x)={\cal R}_Ax+\delta=
$$
\left(\begin{array}{cc}
1  & 1 \\
0  & 1
\end{array}\right)$$
$$
\left(\begin{array}{cc}
x_1 \\
x_2 \\
\end{array}\right)
$$
+
$$
\left(\begin{array}{cc}
\delta_1 \\
\delta_2 \\
\end{array}\right),
$$
\:\,\delta_1 \in \rr, \: \delta_2\in \rr\backslash\{0\}
\qquad \qquad
\end{array}
$$
$$
\begin{array}{cccccc}

\mbox{and} \:\: f_2(x)={\cal R}_Ax+e_2=
$$
\left(\begin{array}{cc}
1  & 1 \\
0  & 1
\end{array}\right)$$
$$
\left(\begin{array}{cc}
x_1 \\
x_2 \\
\end{array}\right)
$$
+
$$
\left(\begin{array}{cc}
0 \\
1 \\
\end{array}\right)
$$
\:\,\mbox {are topologically conjugate},
\qquad \qquad \qquad \qquad
\end{array}
$$
because there is a homeomorphism $\phi_2: \rr^2 \rightarrow \rr^2$ such that $f_2=\phi_2 \circ f_1\circ \phi_2^{-1}$, namely
$$
\begin{array}{cccc}

\phi_2
$$
\left(\begin{array}{cc}
x_1 \\
x_2 \\
\end{array}\right)
$$
=
$$
\left( \begin{array}{ll}
\frac{1}{\delta_2} & -\frac{\delta_1}{\delta_2^2} \\
0                  & \: \: \: \frac{1}{\delta_2}
\end{array}\right)
$$

$$
\left(\begin{array}{cc}
x_1 \\
x_2 \\
\end{array}\right),
$$
\:\, \delta_1 \in \rr, \:\delta_2 \in \rr\backslash\{0\}.
\end{array}
$$

$$
\begin{array}{ccc}

\mbox {\textbf{3.} The maps} \:\:  f_2(x)={\cal R}_Ax+e_2=
$$
\left(\begin{array}{cc}
1  & 1 \\
0  & 1
\end{array}\right)$$
$$
\left(\begin{array}{cc}
x_1 \\
x_2 \\
\end{array}\right)
$$
+
$$
\left(\begin{array}{cc}
0 \\
1 \\
\end{array}\right)
$$
&\qquad \qquad \qquad \qquad \qquad \qquad
\end{array}
$$
$$
\begin{array}{cccc}

\mbox{and} \:\: f_3(x)=x+e_2=
$$
\left(\begin{array}{cc}
x_1 \\
x_2 \\
\end{array}\right)
$$
+
$$
\left(\begin{array}{cc}
0 \\
1 \\
\end{array}\right)
$$
\:\: \mbox {are topologically conjugate, because }
&\qquad \qquad \qquad \qquad \qquad \qquad
\end{array}
$$
there is a homeomorphism $\phi_3: \rr^2 \rightarrow \rr^2$ such that $f_3=\phi_3 \circ f_2 \circ \phi_3^{-1}$, namely
$$
\begin{array}{ccc}

\phi_3
$$
\left( \begin{array}{ll}
x_1 \\
x_2
\end{array}\right)
$$
\:=\:
$$
\left( \begin{array}{ll}
x_1-\frac{1}{2}\,(x_2-\frac{1}{2})^2 \\
\qquad \:\: x_2
\end{array}\right).
$$
&\qquad \qquad
\end{array}
$$
\indent

\textbf{4.} By Proposition~\ref{topol. spryaj. zsyviv} $f_3(x)=x+e_2$ $\stackrel{t}{\sim}$ $g(x)=x+d$, $d\neq 0$. Therefore there is a homeomorphism $\psi$ such that $g=\psi \circ f_3\circ \psi^{-1}$.

\indent

From the steps $\textbf{1}-\textbf{4}$ it follows that the maps $f$, $g: \rr^2 \rightarrow \rr^2$, $f(x)=Ax+b$ such
$$
\begin{array}{cccc}

\mbox{that} \:\,
$$
\left\{
\begin{array}{lllll}
\nexists \: q\in \rr^2: f(q)=q,\\
\lambda^{(1)}_{A}=1=\lambda^{(2)}_{A},\\
A\neq E
\end{array} \right.
$$

\mbox{and} \:\:  g(x)=x+d, ~ d \in \rr^2 \backslash \{0\}\:\: \mbox{are topologically conjugate},
&\qquad &\qquad &\qquad
\end{array}
$$
because there is a homeomorphism $h: \rr^2 \rightarrow \rr^2$, $h(x)= \psi \circ \phi_3 \circ \phi_2 \circ \phi_1(x)$ such that $g=h \circ f\circ h^{-1}$ (where $\phi_1$, $\phi_2$, $\phi_3$, $\psi$ are the corresponding homeomorphisms from the steps $\textbf{1}-\textbf{4}$\,).

\indent

\textbf{Case 2: $\lambda^{(1)}_A=1$ and $\lambda^{(2)}_A=\alpha$, where $\alpha \in \rr \backslash\{0,1\}$.}
$$
\begin{array}{cccc}

\mbox{We will prove that}\:\, f(x)=Ax+b \:\,\mbox{such that}\:\,
$$
\left\{
\begin{array}{ll}
\nexists \: q\in \rr^2: f(q)=q,\\
\lambda^{(1)}_{A}=1, \:\lambda^{(2)}_{A}=\alpha, \:\mbox {where} \:\: \alpha \in \rr \backslash\{0,1\}
\end{array} \right.
$$
&\qquad &\qquad &\qquad
\end{array}
$$
and $g(x)=x+d$, $d \in \rr^2\backslash\{0\}$ are topologically conjugate.

The idea of the proof of the Case 2 is similar to the previous one, but the corresponding homeomorphisms are different.

We will construct a homeomorphism $h: \rr^2 \rightarrow \rr^2$ such that $g=h \circ f\circ h^{-1}$ in four steps.
$$
\begin{array}{ccc}

\mbox {\textbf{1'.} Let}\:\, f(x)=Ax+b \:\, \mbox {be such that}\:\,
$$
\left\{
\begin{array}{ll}
\nexists \: q\in \rr^2: f(q)=q,\\
\lambda^{(1)}_A=1,\: \lambda^{(2)}_A=\alpha, \: \alpha \in \rr \setminus\{0,1\}
\end{array} \right.
$$
\:\mbox{and}
\qquad \quad
\end{array}
$$
$$
\begin{array}{cccc}

\phi_1: \rr^2 \rightarrow \rr^2,\:\,  \phi_1(x)=Bx, \: \mbox {where} \: B: \:
$$
\left\{
\begin{array}{ll}
\det B\neq 0, \\
BAB^{-1}={\cal R}_A, \\
\end{array}\right.
$$
\:\, \mbox {then}\:\, f_1=\phi_1 \circ f\circ \phi_1^{-1}, \:\, \mbox {where}\:\,
\qquad
\end{array}
$$
$$
\begin{array}{ccc}

f_1(x)={\cal R}_Ax+\delta=
$$
\left(\begin{array}{cc}
1  & 0      \\
0  & \alpha \\
\end{array}\right)$$
$$
\left(\begin{array}{cc}
x_1 \\
x_2 \\
\end{array}\right)
$$
+
$$
\left(\begin{array}{cc}
\delta_1 \\
\delta_2 \\
\end{array}\right),
$$
\:\, \alpha \in \rr \backslash\{0,1\}, \:\delta=Bb,\: \delta_1 \in \rr \backslash\{0\},\: \delta_2 \in \rr.
\qquad \qquad \qquad \qquad
\end{array}
$$

Conditions $\delta_1 \in \rr \backslash\{0\}$, $\delta_2 \in \rr$ follows from Proposition~\ref{ne-isnye-neryx.-toch-1,d}, because the map $f_1(x)={\cal R}_Ax+\delta$ does not have any fixed point ($f_1 \stackrel{t}{\sim} f$).

$$
\begin{array}{cccc}

\mbox {\textbf{2'.} The maps} \:\:  f_1(x)={\cal R}_Ax+\delta=
$$
\left(\begin{array}{cc}
1  & 0      \\
0  & \alpha \\
\end{array}\right)$$
$$
\left(\begin{array}{cc}
x_1 \\
x_2 \\
\end{array}\right)
$$
+
$$
\left(\begin{array}{cc}
\delta_1 \\
\delta_2 \\
\end{array}\right)
$$
\:\, \mbox {and}
&\qquad &\qquad \qquad \qquad
\end{array}
$$
$$
\begin{array}{cccc}

f_2(x)={\cal R}_Ax+e_1=
$$
\left(\begin{array}{cc}
1  & 0      \\
0  & \alpha \\
\end{array}\right)$$
$$
\left(\begin{array}{cc}
x_1 \\
x_2 \\
\end{array}\right)
$$
+
$$
\left(\begin{array}{cc}
1 \\
0 \\
\end{array}\right),
$$
\:\, \alpha \in \rr \backslash\{0,1\},\: \delta_1 \in \rr \backslash\{0\},\: \delta_2 \in \rr
\qquad \qquad
\end{array}
$$
are topologically conjugate, because there is a homeomorphism $\phi_2: \rr^2 \rightarrow \rr^2$ such that $f_2=\phi_2 \circ f_1\circ \phi_2^{-1}$, namely
$$
\begin{array}{cc}

\phi_2
$$
\left( \begin{array}{ll}
x_1 \\
x_2
\end{array}\right)
$$
=
$$
\left( \begin{array}{ll}
\quad  \frac{x_1}{\delta_1} \\
x_2+\frac{\delta_2}{\alpha -1}
\end{array}\right),
$$
\:\, \alpha \in \rr \backslash\{0,1\},\, \delta_1 \in \rr \backslash\{0\},\, \delta_2 \in \rr.
\end{array}
$$

$$
\begin{array}{cccc}

\mbox {\textbf{3'.} The maps} \:\:  f_2(x)={\cal R}_Ax+e_1=
$$
\left(\begin{array}{cc}
1  & 0      \\
0  & \alpha \\
\end{array}\right)$$
$$
\left(\begin{array}{cc}
x_1 \\
x_2 \\
\end{array}\right)
$$
+
$$
\left(\begin{array}{cc}
1 \\
0 \\
\end{array}\right),
$$
\:\, \alpha \in \rr \backslash\{0,1\}
\qquad \qquad \qquad
\end{array}
$$
$$
\begin{array}{cccc}

\mbox {and}\:\, f_3(x)=x+e_1=
$$
\left(\begin{array}{cc}
x_1 \\
x_2 \\
\end{array}\right)
$$
+
$$
\left(\begin{array}{cc}
1 \\
0 \\
\end{array}\right)
$$
\:\, \mbox {are topologically conjugate, because there is}
\qquad \qquad \qquad \qquad
\end{array}
$$
a homeomorphism $\phi_3: \rr^2 \rightarrow \rr^2$ such that $f_3=\phi_3 \circ f_2\circ \phi_3^{-1}$, namely
$$
\begin{array}{cccc}

\phi_3
$$
\left( \begin{array}{ll}
x_1 \\
x_2
\end{array}\right)
$$
=
$$
\left( \begin{array}{ll}
\quad  x_1 \\
x_2 \alpha^{- x_1}
\end{array}\right),
$$
\:\: \alpha \in \rr \backslash\{0,1\}.
\qquad & \qquad
\end{array}
$$


\textbf{4'.} By Proposition~\ref{topol. spryaj. zsyviv} $f_3(x)=x+e_1$ $\stackrel{t}{\sim}$ $g(x)=x+d$, $d\neq 0$. Therefore there is a homeomorphism $\psi$ such that $g=\psi \circ f_3\circ \psi^{-1}$.


\indent

From the steps $\textbf{1'}-\textbf{4'}$ it follows that the maps $f,\, g: \rr^2 \rightarrow \rr^2$, $f(x)=Ax+b$ such
$$
\begin{array}{cccc}

\mbox{that} \:\,
$$
\left\{
\begin{array}{ll}
\nexists \: q\in \rr^2\:: f(q)=q,\\
\lambda^{(1)}_{A}=1, \:\lambda^{(2)}_{A}=\alpha, \: \mbox {where} \:\, \alpha \in \rr \backslash\{0,1\}
\end{array} \right.
$$
\mbox {and} \:\:  g(x)=x+d,\:\, d \in \rr^2 \backslash \{0\} \:\: \mbox {are}
&\qquad &\qquad &\qquad
\end{array}
$$
topologically conjugate, because there is a homeomorphism $h: \rr^2 \rightarrow \rr^2$, $h(x)= \psi \circ \phi_3 \circ \phi_2 \circ \phi_1(x)$ such that $g=h \circ f\circ h^{-1}$ (where $\phi_1$, $\phi_2$, $\phi_3$, $\psi$ are the corresponding homeomorphisms from the steps  $\textbf{1'}-\textbf{4'}$\,).

\indent

\textbf{Conclusion:}
$$
\begin{array}{cccc}

\mbox {From the Case 1 it follows that} \: f(x)=Ax+b \:\: \mbox {such that}\:
$$
\left\{\begin{array}{ll}
\nexists \: q \in \rr^2: f(q)=q,\\
\lambda^{(1)}_{A}=\lambda^{(2)}_{A}=1
\end{array}\right.
$$
\qquad \qquad
\end{array}
$$
and $g(x)=x+d$, $d \in \rr^2\backslash\{0\}$\, are topologically conjugate.
$$
\begin{array}{cc}

\mbox {From the Case 2 it follows that} \: f(x)=Ax+b \:\, \mbox {such that}\,
$$
\left\{\begin{array}{ll}
\nexists \: q \in \rr^2: f(q)=q, \\
\lambda^{(1)}_{A}=1, \,\lambda^{(2)}_{A}=\alpha, \,\alpha \in \rr \backslash\{0,1\}
\end{array}\right.
$$
\end{array}
$$
and $g(x)=x+d$, $d \in \rr^2\backslash\{0\}$\, are topologically conjugate.

$$
\begin{array}{cccc}

\mbox {Consequently, a map} \: f(x)=Ax+b \:\: \mbox {such that}\:
$$
\left\{\begin{array}{ll}
\nexists \: q \in \rr^2: f(q)=q, \\
\det A\neq 0 \:( \Leftrightarrow \lambda^{(2)}_{A}\neq 0)
\end{array}\right.
$$
\qquad \qquad \qquad \:
\end{array}
$$
and $g(x)=x+d$, $d\in \rr^2\backslash\{0\}$\, are topologically conjugate.

\end{proof}

\begin{nas}\cite{st-2}\label{detA-ne-0, detC-ne-0}
Let $f$, $g: \rr^2 \rightarrow \rr^2$, $f(x)=Ax+b$, $g(x)=Cx+d$ be affine maps such that:
\begin{enumerate}
  \item[1)] $f$ and $g$ have no fixed points;
  \item[2)] $\det A\neq0$ and $\det C\neq0$.
\end{enumerate}
Then $f\stackrel{t}{\sim}g$.
\end{nas}

\indent

\textbf{b) Topological classification of non-bijective affine maps, which don't have fixed points.}

The following Theorem~\ref{detA=0, detC=0} asserts that two non-bijective affine maps, which don't have fixed points are topologically conjugate.

\begin{theorem}\cite{st-2}\label{detA=0, detC=0}
Let $f$, $g: \rr^2 \rightarrow \rr^2$, $f(x)=Ax+b$, $g(x)=Cx+d$ be affine maps such that:
\begin{enumerate}
  \item[1)] $f$ and $g$ have no fixed points;
  \item[2)] $\det A=0$ and $\det C=0$.
\end{enumerate}
Then $f\stackrel{t}{\sim}g$.
\end{theorem}

\begin{proof}
The idea of the proof is similar to the proof of the previous theorem. For maps $f$ and $g$, that satisfy the conditions of this theorem, we will construct a homeomorphism $h: \rr^2 \rightarrow \rr^2$ such that $g=h \circ f\circ h^{-1}$.

By using Lemma~\ref{isnye-lambda_(A)=1} and Proposition~$\ref{ne-isnye-neryx.-toch-1,d}$ we will prove that $f(x)=Ax+b$ and
$$
\begin{array}{cc}

g(x)=Cx+d \:\,\mbox {such that} \:\:
$$
\left\{
\begin{array}{ll}
\nexists \: \:q, \nu \in \rr^2: f(q)=q \:\: \mbox {and}\:\: g(\nu)= \nu,\\
\lambda^{(1)}_{A}=\lambda^{(1)}_{C}=1,\: \lambda^{(2)}_{A}=\lambda^{(2)}_{C}=0
\end{array} \right.
$$
\:\mbox {are topologically}
\qquad \qquad \qquad
\end{array}
$$
conjugate. Construction of the corresponding homeomorphism is divided into three steps.

$$
\begin{array}{cc}

\mbox {\textbf{1.} Let}\:\,f(x)=Ax+b \:\: \mbox{be such that}\:\:
$$
\left\{
\begin{array}{ll}
\nexists \: \:q \in \rr^2: f(q)=q,\\
\lambda^{(1)}_{A}=1,\: \lambda^{(2)}_{A}=0
\end{array} \right.
$$
\:\,\mbox{and homeomorphism}
\qquad \qquad
\end{array}
$$
$$
\begin{array}{cc}

\phi_1: \rr^2 \rightarrow \rr^2,\:\,  \phi_1(x)=Sx, \: \mbox {where matrix} \: S: \:
$$
\left\{ \begin{array}{ll}
\det S\neq 0, \\
SAS^{-1}={\cal R}_A, \\
\end{array}\right.
$$
\mbox{then} \:\,f_1=\phi_1 \circ f\circ \phi_1^{-1},
\qquad \qquad
\end{array}
$$
$$
\begin{array}{cc}

\mbox{where} \:\,f_1(x)={\cal R}_Ax+\delta=
$$
\left(\begin{array}{cc}
1 & 0 \\
0 & 0 \\
\end{array}\right)$$
$$
\left(\begin{array}{cc}
x_1 \\
x_2 \\
\end{array}\right)
$$
+
$$
\left(\begin{array}{cc}
\delta_1 \\
\delta_2 \\
\end{array}\right),
$$
\:\,\delta=Sb, \: \delta_1\in \rr\backslash\{0\},\: \delta_2 \in \rr.
\qquad & \qquad
\end{array}
$$

Conditions $\delta_1\in \rr\backslash\{0\}$, $\delta_2 \in \rr$ follows from Proposition~\ref{ne-isnye-neryx.-toch-1,d}, because the map $f_1(x)={\cal R}_Ax+\delta$ does not have any fixed point ($f_1 \stackrel{t}{\sim} f$).

$$
\begin{array}{ll}

\mbox {\textbf{2.} Let}\:\, g(x)=Cx+d \:\: \mbox {be such that}\:\:
$$
\left\{
\begin{array}{ll}
\nexists \: \:\nu \in \rr^2: g(\nu)=\nu,\\
\lambda^{(1)}_{C}=1, \:\lambda^{(2)}_{C}=0
\end{array} \right.
$$
\:\:\mbox{and homeomorphism}
\qquad \quad
\end{array}
$$

$$
\begin{array}{ll}

\phi_3: \rr^2 \rightarrow \rr^2,\:\,  \phi_3(x)=Kx, \: \mbox {where matrix} \: K: \:
$$
\left\{ \begin{array}{ll}
\det K\neq 0, \\
K^{-1}CK={\cal R}_C={\cal R}_A, \\
\end{array}\right.
$$
\mbox{then}\:\, f_2=\phi_3^{-1} \circ g\circ \phi_3\,,
\qquad
\end{array}
$$
$$
\begin{array}{ll}

\mbox {where}\:\, f_2(x)={\cal R}_Ax+\eta=
$$
\left(\begin{array}{ll}
1 & 0 \\
0 & 0 \\
\end{array}\right)$$
$$
\left(\begin{array}{ll}
x_1 \\
x_2 \\
\end{array}\right)
$$
+
$$
\left(\begin{array}{ll}
\eta_1 \\
\eta_2 \\
\end{array}\right),
$$
\:\,\eta=K^{-1}d,\: \eta_1\in \rr\backslash\{0\},\: \eta_2 \in \rr.
\qquad \qquad \qquad \qquad
\end{array}
$$

Conditions $\eta_1\in \rr\backslash\{0\}$, $\eta_2 \in \rr$ follows from Proposition~\ref{ne-isnye-neryx.-toch-1,d}, because the map $f_2(x)={\cal R}_Ax+\eta$ does not have any fixed point ($f_2 \stackrel{t}{\sim} g$).

$$
\begin{array}{cccc}

\mbox {\textbf{3.} The maps} \:\:  f_1(x)={\cal R}_Ax+\delta=
$$
\left(\begin{array}{cc}
1 & 0 \\
0 & 0 \\
\end{array}\right)$$
$$
\left(\begin{array}{cc}
x_1 \\
x_2 \\
\end{array}\right)
$$
+
$$
\left(\begin{array}{cccc}
\delta_1 \\
\delta_2 \\
\end{array}\right),
$$
\:\, \delta_1\in \rr\backslash\{0\},\: \delta_2 \in \rr
\qquad \qquad
\end{array}
$$
$$
\begin{array}{cc}

\mbox {and} \:\,f_2(x)={\cal R}_Ax+\eta=
$$
\left(\begin{array}{cc}
1 & 0 \\
0 & 0
\end{array}\right)
$$

$$
\left(\begin{array}{cc}
x_1 \\
x_2
\end{array}\right)
$$
+
$$
\left(\begin{array}{cc}
\eta_1 \\
\eta_2
\end{array}\right),
$$
\:\, \eta_1\in \rr\backslash\{0\},\: \eta_2 \in \rr
\qquad \qquad \qquad \qquad
\end{array}
$$
are topologically conjugate, because there is a homeomorphism $\phi_2: \rr^2 \rightarrow \rr^2$ such that $f_2=\phi_2 \circ f_1\circ \phi_2^{-1}$, namely
$$
\begin{array}{cccc}

\phi_2
$$
\left( \begin{array}{ll}
x_1 \\
x_2
\end{array}\right)
$$
=
$$
\left( \begin{array}{ll}
\quad \: \frac{\eta_1}{\delta_1}x_1 \\
x_2-\delta_2+\eta_2
\end{array}\right),
$$
\:\, \delta_1, \,\eta_1\in \rr\backslash\{0\},\: \delta_2,\,\eta_2 \in \rr.
\end{array}
$$

\indent

From the steps $\textbf{1}-\textbf{3}$ it follows that $f$, $g: \rr^2 \rightarrow \rr^2$, $f(x)=Ax+b$ and $g(x)=Cx+d$
$$
\begin{array}{lllllll}

\mbox {such that} \:\,
$$
\left\{
\begin{array}{ll}
\nexists \: \:q, \nu \in \rr^2: f(q)=q \:\: \mbox {and}\:\: g(\nu)= \nu,\\
\lambda^{(1)}_{A}=\lambda^{(1)}_{C}=1, \:\lambda^{(2)}_{A}=\lambda^{(2)}_{C}=0
\end{array} \right.
$$
\:\mbox {are topologically conjugate}, \:\mbox {because}
&\qquad &\qquad &\qquad &\qquad &\qquad
\end{array}
$$
there is a homeomorphism $h: \rr^2 \rightarrow \rr^2$, $h(x)= \phi_3 \circ \phi_2 \circ \phi_1(x)$ such that $g=h \circ f\circ h^{-1}$,
(where $\phi_1$, $\phi_2$, $\phi_3$ are the corresponding homeomorphisms from the steps $\textbf{1}-\textbf{3}$\,).
\end{proof}

\indent

Consequently, by Corollary~\ref{detA-ne-0, detC-ne-0} and Theorem~\ref{detA=0, detC=0} for affine maps from $\rr^2$ to $\rr^2$, which have no fixed points we get the next proposition.

\begin{prop}\cite{st-2}\label{osnovna--bez-neryx-tochok}
Let $f$, $g: \rr^2 \rightarrow \rr^2$, $f(x)=Ax+b$, $g(x)=Cx+d$ be affine maps, which have no fixed points.

Then $f\stackrel{t}{\sim}g$ if and only if \,$\det A$ and $\det C$ are either simultaneously equal to 0 or simultaneously different from 0.
\end{prop}

\indent

The following theorem is a conclusion of Propositions~\ref{klas-affin-vidob-2-2} and \ref{osnovna--bez-neryx-tochok}.

\begin{theorem}\cite{st-2}\:\,[Main theorem]\label{osnovna-R-2}
Let $f$, $g: \rr^2 \rightarrow \rr^2$, $f(x)=Ax+b$, $g(x)=Cx+d$ be affine maps.

If each of $f$ and $g$ has at least one fixed point, then $f\stackrel{t}{\sim}g$ if and only if
$$
\left\{
\begin{array}{ll}
\rank(A_+)=\rank(C_+), \: \sign(\,\det(A_+)\,)=\sign(\,\det(C_+)\,),\\
\rank(A_-)=\rank(C_-), \: \sign(\,\det(A_-)\,)=\sign(\,\det(C_-)\,), \\
A_\infty = C_\infty, \\
A_0 = C_0.
\end{array} \right.
$$

If $f$ and $g$ have no fixed points, then $f\stackrel{t}{\sim}g$ if and only if \,$\det A$ and $\det C$ are either simultaneously equal to 0 or simultaneously different from 0.
\end{theorem}


\indent

\section{Topological classification of affine maps from $\cc^n$ to $\cc^n$, $n\leq 2$.}

The problem of classification up to topological conjugacy of affine maps from $\cc^n$ to $\cc^n$, $n\geq 1$ is still open, but we obtained the classification of affine maps from $\cc^n$ to $\cc^n$, $n=1,\, 2$.

\subsection{Topological classification of affine maps from $\cc$ to $\cc$.}

\subsubsection{Classification of linear maps from $\cc$ to $\cc$.}


\begin{lemma}\cite{st-3}\label{klas-liniu-vidob-C}
Let $f$, $g: \cc \rightarrow \cc$, $f(z)=az$, $g(z)=cz$, $a$, $c \in\cc$ be linear maps.
$$
\begin{array}{ll}

\mbox{Then}\:\, f \stackrel{t}{\sim}g \:\, \mbox{if and only if} \:\: \mbox{or}\:\,
$$
\left\{
\begin{array}{cc}
0<|a|<1,  \\
0<|c|<1,
\end{array}\right.
$$
\,\mbox{or}\:\,
$$
\left\{
\begin{array}{cc}
|a|>1,  \\
|c|>1,
\end{array}\right.
$$
\,\mbox{or}\:\: a=c, \:\,\mbox{or}\:\: a=\bar{c}.
\qquad \qquad \qquad \qquad \qquad
\end{array}
$$
(That is when  $|a|$, $|c|$ are or simultaneously smaller than 1 and different from 0, or simultaneously greater than~1, or $a=c$, or  $a=\bar{c}$).
\end{lemma}

\begin{proof}

The maps $f$, $g: \cc \rightarrow \cc$, $f(z)=az$, $g(z)=cz$, $a=a_1+ia_2$, $c=c_1+ic_2 \in\cc$ are topologically conjugate if and only if the corresponding linear maps from $\rr^2$ to $\rr^2$ are topologically conjugate, that is when
$$
\begin{array}{ll}

f_{\rr^2}(x)=Ax=
$$
\left(
\begin{array}{cc}
a_1 & -a_2   \\
a_2 & a_1
\end{array}\right)
$$
$$
\left(
\begin{array}{cc}
x_1 \\
x_2
\end{array}\right),
$$
\:\,\mbox{where} ~ a_1, ~ a_2 \in \rr
\:\,\mbox{and} \:\, \lambda^{(1)}_{A}=a,\, \lambda^{(2)}_{A}=\bar{a}
\qquad \qquad \qquad \qquad
\end{array}
$$
$$
\begin{array}{ll}

\mbox{and}~ g_{\rr^2}(x)=Cx=
$$
\left(
\begin{array}{cc}
c_1 & -c_2   \\
c_2 & c_1
\end{array}\right)
$$
$$
\left(
\begin{array}{cc}
x_1 \\
x_2
\end{array}\right),
$$
\:\, \mbox{where}~ c_1, ~ c_2 \in \rr
\:\,\mbox{and} \:\, \lambda^{(1)}_C=c,\, \lambda^{(2)}_C=\bar{c}
\qquad \qquad \qquad \qquad
\end{array}
$$
are topologically conjugate.

By Proposition~\ref{klas-vsix-liniu-vidob-5-5} $f_{\rr^2}(x)=Ax$ and $g_{\rr^2}(x)=Cx$ are topologically conjugate
$$
\begin{array}{ll}

\mbox{if and only if} \:\:
$$
\left\{
\begin{array}{ll}
\rank(A_+)=\rank(C_+), \: \sign(\,\det(A_+)\,)=\sign(\,\det(C_+)\,),\\
\rank(A_-)=\rank(C_-), \: \sign(\,\det(A_-)\,)=\sign(\,\det(C_-)\,), \\
A_\infty = C_\infty, \\
A_0 = C_0.
\end{array} \right.
$$

\qquad \qquad \qquad \qquad \qquad
\end{array}
$$

Notice, that the RCF's of the matrices $A$ and $C$ coincide with the matrices $A$ and $C$, respectively, and these matrices have the following properties:

1) $\rank(A)=2=\rank(C)$ for all non-zero matrices $A$ and $C$;

2) $\sign(\,\det(A)\,)=\sign\,(\,a_1^2+a_2^2\,)=\sign\,(\,c_1^2+c_2^2\,)=\sign(\,\det(C)\,)$ for all non-zero matrices $A$ and $C$;

3) $|\lambda^{(1)}_A|=|\lambda^{(2)}_A|=|a|$, $|\lambda^{(1)}_C|=|\lambda^{(2)}_C|=|c|$;

4) $a=c=0$ if and only if $A$ and $C$ are zero matrices;

5) if $|a|=|c|=1$, then $f_{\rr^2}(x)=Ax$ $\stackrel{t}{\sim}$ $g_{\rr^2}(x)=Cx$ if and only if either $a=c$, or $a=\bar{c}$.

The proof of 5-th property follows from Proposition~\ref{klas-vsix-liniu-vidob-5-5}, because if $|a|=|c|=1$ then $f_{\rr^2}(x)=Ax$ $\stackrel{t}{\sim}$ $g_{\rr^2}(x)=Cx$ if and only if $A_{0}=C_{0}$. From the equality of these matrices follows that their eigenvalues are identical, therefore $\{\lambda^{(1)}_{A}, \lambda^{(2)}_{A}\}$=$\{\lambda^{(1)}_{C}, \lambda^{(2)}_{C}\}$. As $\{\lambda^{(1)}_{A}, \lambda^{(2)}_{A}\}$=$\{a, \bar{a}\}$ and  $\{\lambda^{(1)}_{C}, \lambda^{(2)}_{C}\}$=$\{c, \bar{c}\}$, it follows that either $a=c$, or $a=\bar{c}$.
$$
\begin{array}{ll}

\mbox{\quad Conversely, if}\:\, a=c,\:\, \mbox{then}\:\, A=C; \:\: \mbox{if}\:\,  a=\bar{c}, \:\,\mbox{then}\:\, f_{\rr^2}(x)=Ax=
$$
\left(
\begin{array}{cc}
a_1 & -a_2   \\
a_2 & a_1
\end{array}\right)
$$
$$
\left(
\begin{array}{cc}
x_1 \\
x_2
\end{array}\right)
$$
\qquad \qquad \qquad \qquad
\end{array}
$$
$$
\begin{array}{ll}
\mbox{and}~ g_{\rr^2}(x)=Cx=
$$
\left(
\begin{array}{cc}
a_1  & a_2   \\
-a_2 & a_1
\end{array}\right)
$$
$$
\left(
\begin{array}{cc}
x_1 \\
x_2
\end{array}\right)
$$
\:\,\mbox{are topologically conjugate, because there is}
\qquad \qquad \qquad \qquad
\end{array}
$$
$$
\begin{array}{ll}

\mbox{a homeomorphism}\:\,
h: \rr^2 \rightarrow \rr^2,  \:\,h
$$
\left(
\begin{array}{cc}
x_1 \\
x_2
\end{array}\right)
$$
=
$$
\left(
\begin{array}{cc}
1 & 0   \\
0 & -1
\end{array}\right)
$$
$$
\left(
\begin{array}{cc}
x_1 \\
x_2
\end{array}\right),
$$
\:\, \mbox{such that}\:\, f=h\circ g\circ h^{-1}.
\qquad \qquad
\end{array}
$$

By the properties $1 - 5$, we obtain that the maps
$$
\begin{array}{ll}

f_{\rr^2}(x)=Ax=
$$
\left(
\begin{array}{cc}
a_1 & -a_2   \\
a_2 & a_1
\end{array}\right)
$$
$$
\left(
\begin{array}{cc}
x_1 \\
x_2
\end{array}\right),
$$
\:\,\mbox{where} ~ a_1, ~ a_2 \in \rr,
\:\, \lambda^{(1)}_{A}=a,\, \lambda^{(2)}_{A}=\bar{a}
\qquad \qquad \qquad \qquad
\end{array}
$$
$$
\begin{array}{ll}

\mbox{and}~ g_{\rr^2}(x)=Cx=
$$
\left(
\begin{array}{cc}
c_1 & -c_2   \\
c_2 & c_1
\end{array}\right)
$$
$$
\left(
\begin{array}{cc}
x_1 \\
x_2
\end{array}\right),
$$
\:\, \mbox{where}~ c_1, ~ c_2 \in \rr,
\:\, \lambda^{(1)}_C=c,\, \lambda^{(2)}_C=\bar{c}\:\: \mbox{are}
\qquad \qquad \qquad
\end{array}
$$
$$
\begin{array}{ll}

\mbox{topologically conjugate}\:\,\mbox{if and only if} \:\, \mbox{or}\:\,
$$
\left\{
\begin{array}{cc}
0<|a|<1,  \\
0<|c|<1,
\end{array}\right.
$$
\,\mbox{or}\:\,
$$
\left\{
\begin{array}{cc}
|a|>1,  \\
|c|>1,
\end{array}\right.
$$
\: \mbox{or}\:\, a=c, \:\,\mbox{or}\:\, a=\bar{c}.
\qquad \qquad \qquad \qquad \qquad
\end{array}
$$

\indent

Therefore $f$, $g: \cc \rightarrow \cc$, $f(z)=az$, $g(z)=cz$, $a$, $c \in \cc$ are topologically conjugate
$$
\begin{array}{ll}

\mbox{if and only if} \:\,\mbox{or}\:\,
$$
\left\{
\begin{array}{cc}
0<|a|<1,  \\
0<|c|<1,
\end{array}\right.
$$
\:\mbox{or}\:\,
$$
\left\{
\begin{array}{cc}
|a|>1,  \\
|c|>1,
\end{array}\right.
$$
\:\mbox{or}\:\, a=c, \:\,\mbox{or}\:\, a=\bar{c}.
\qquad \qquad \qquad \qquad \qquad
\end{array}
$$
\end{proof}

\begin{remk}
The 5-th property in the proof of Lemma~\ref{klas-liniu-vidob-C} is quite natural, because the RCF of a real entried $n\times n$ matrix is uniquely determined up to sign of non-diagonal elements and up to the order of diagonal blocks (in the cases $n>2$).
\end{remk}


\indent
\subsubsection{Classification of affine maps from $\cc$ to $\cc$.}

\begin{theorem}\cite{st-3}\label{pro-neryx-tochk-C}
Let $f$, $g: \cc\rightarrow \cc$, $f(z)=az+b$ and $g(z)=az$, $a$, $b \in\cc$.

Then $f\stackrel{t}{\sim}g$ if and only if there is $q \in \cc$ such that $f(q)=q$.
\end{theorem}

\begin{proof}
Similar to the proof of Theorem~\ref{pro-neryx-tochk}.
\end{proof}

The following theorem is the criterion of topological conjugacy of affine maps from $\cc$ to $\cc$.

\begin{theorem}\cite{st-3}\label{class-affin-vidobr}
Let $f$, $g: \cc \rightarrow \cc$, $f(z)=az+b$, $g(z)=cz+d$, $a$, $b$, $c$, $d \in\cc$ be affine maps.
$$
\begin{array}{ll}

\:\,\mbox{Then}\:\, f \stackrel{t}{\sim}g \:\, \mbox{if and only if} \:\,\mbox{or}\:\,
$$
\left\{
\begin{array}{cc}
0<|a|<1,  \\
0<|c|<1,
\end{array}\right.
$$
\:\mbox{or}\:\,
$$
\left\{
\begin{array}{cc}
|a|>1,  \\
|c|>1,
\end{array}\right.
$$
\:\mbox{or} \:\: a=c,\:\,\mbox{or}\:\, a=\bar{c} \,;\,\,\mbox{if}
\qquad \qquad \qquad \qquad
\end{array}
$$
$a=c=1$, then $b$ and $d$ are either simultaneously equal to 0 or simultaneously different from 0.
\end{theorem}

\begin{proof}

The topologically conjugate maps have the same number of fixed points, therefore we will prove this theorem for each class of affine maps in which the number of fixed points is the same.

Since an arbitrary affine map from $\cc$ to $\cc$ has either one fixed point, or infinitely many fixed points (that is the identity), or doesn't have them at all, it follows that we have three cases:

1) The affine maps such that each of them has only one fixed point.

An arbitrary affine map from $\cc$ to $\cc$, $\varphi(z)=mz+n$ has only one fixed point, (namely $z=\frac{n}{1-m}$) if and only if $m\neq1$, $n \in \cc$.

By Theorem~\ref{pro-neryx-tochk-C}: \,$f(z)=az+b \stackrel{t}{\sim} r(z)=az$, $a \in \cc \backslash\{1\}$, $b \in \cc$,

\qquad \qquad \qquad \qquad           $g(x)=cz+d \stackrel{t}{\sim} s(z)=cz$, $c \in \cc \backslash\{1\}$, $d \in \cc$.

By Lemma~\ref{klas-liniu-vidob-C}: \,$r(z)=az \stackrel{t}{\sim} s(z)=cz$, $a$, $c \in \cc \backslash\{1\}$ if and only if
$$
\begin{array}{l}

\mbox{or}\:\,
$$
\left\{
\begin{array}{cc}
0<|a|<1,  \\
0<|c|<1,
\end{array}\right.
$$
\:\,\mbox{or}\:\,
$$
\left\{
\begin{array}{cc}
|a|>1,  \\
|c|>1,
\end{array}\right.
$$
\:\,\mbox{or}\:\: a=c, \:\,\mbox{or}\:\: a=\bar{c}.
\qquad \qquad \qquad \qquad \qquad \qquad \qquad \qquad
\end{array}
$$

Consequently, $f(z)=az+b$ $\stackrel{t}{\sim}$ $g(z)=cz+d$, $a$, $c \in \cc \backslash\{1\}$, $b$,~$d \in \cc$\, if and only if
$$
\begin{array}{ll}

\mbox{or}\:\,
$$
\left\{
\begin{array}{cc}
0<|a|<1,  \\
0<|c|<1,
\end{array}\right.
$$
\:\mbox{or}\:\,
$$
\left\{
\begin{array}{cc}
|a|>1,  \\
|c|>1,
\end{array}\right.
$$
\:\,\mbox{or}\:\: a=c,\:\,\mbox{or}\:\: a=\bar{c}.
\qquad \qquad \qquad \qquad \qquad \qquad \qquad \qquad
\end{array}
$$

2) The affine maps such that each of them has infinitely many fixed points.

An arbitrary affine map from $\cc$ to $\cc$, $\varphi(z)=mz+n$ has infinitely many fixed points if and only if $m=1$ and $n=0$, that is, $\varphi(z)=id_{\cc}(z)$.

3) The affine maps, which have no fixed points.

An arbitrary affine map from $\cc$ to $\cc$, $\varphi(z)=mz+n$ has no fixed point if and only if $m=1$ and $n\neq0$.

The maps $f(z)=z+b$ and $g(z)=z+d$, $b$, $d \in \cc \backslash\{0\}$, which have no fixed points are always topologically conjugate, because there is a homeomorphism  $h: \cc \rightarrow \cc$, $h(z)=\frac{d}{b}\,z$, $b$, $d \in \cc \backslash\{0\}$ such that $h\circ f=g\circ h$.

If we combine the results of these three cases, then
$f(z)=az+b$ $\stackrel{t}{\sim}$ $g(z)=cz+d$, 
$$
\begin{array}{ll}

a,~c \in \cc \backslash\{1\},\:\, b,~d \in \cc \:\, \mbox{if and only if} \:\,
\mbox{or}\:\,
$$
\left\{
\begin{array}{cc}
0<|a|<1,  \\
0<|c|<1,
\end{array}\right.
$$
\:\mbox{or}\:\,
$$
\left\{
\begin{array}{cc}
|a|>1,  \\
|c|>1,
\end{array}\right.
$$
\:\,\mbox{or}\:\: a=c,\:\,\mbox{or}\:\: a=\bar{c}.
\qquad \qquad \qquad \qquad \qquad \qquad \qquad \qquad
\end{array}
$$
If $a=c=1$, then $f(z)=z+b$ $\stackrel{t}{\sim}$ $g(z)=z+d$, $b$, $d \in \cc$ if and only if $b$ and $d$ are either simultaneously equal to 0 or simultaneously different from 0.

Consequently, $f(z)=az+b$ $\stackrel{t}{\sim}$ $g(z)=cz+d$, $a$, $c$, $b$,~$d \in \cc$\, if and only if
$$
\begin{array}{lll}

\mbox{or}\:\,
$$
\left\{
\begin{array}{cc}
0<|a|<1,  \\
0<|c|<1,
\end{array}\right.
$$
\:\,\mbox{or}\:\,
$$
\left\{
\begin{array}{cc}
|a|>1,  \\
|c|>1,
\end{array}\right.
$$
\:\,\mbox{or}\:\: a=c,\:\,\mbox{or}\:\, a=\bar{c}\,;\:\, \mbox{if}\:\:a=c=1,

&\qquad \qquad \qquad \qquad \qquad
\end{array}
$$
then $b$ and $d$ are either simultaneously equal to 0 or simultaneously different from 0.

\end{proof}

\indent
\subsection{Topological classification of affine maps from $\cc^2$ to $\cc^2$.}

\subsubsection{Classification of linear maps from $\cc^2$ to $\cc^2$.}

Denote by $J_k (\lambda)$ the $k \times k$ \emph{Jordan block} with eigenvalue $\lambda$:
$$
J_k (\lambda)=
\left(
\begin{array}{lllll}
\lambda & 1       &            & \, 0  \\
        & \ddots  & \ddots     &       \\
        &         & \lambda    & \,1   \\
0       &         &            & \lambda
\end{array} \right).
$$
The number $k$ is called the \emph{size of $J_k (\lambda)$}.

For each square complex matrix $A$, Jordan canonical form $J_A$ is a direct sum of Jordan blocks.

\begin{ozn}
We say that two Jordan canonical forms
\[
J_A=J_{k_1}(\lambda_1)\oplus \ldots \oplus J_{k_t} (\lambda_t) \:\, \mbox{and}\:\,
J_C=J_{k_1}(\mu_1)\oplus \ldots \oplus J_{k_t} (\mu_t)
\]
are \textbf{*equal} (and write $J_A \stackrel{\ast}{=} J_C$)\, if $\mu_i = \lambda_i$ or $\mu_i = \overline{\lambda_i}$\, for all $i=1,\dots, t$.
\end{ozn}
\begin{exam}
$$
\begin{array}{l}

$$
\left(
\begin{array}{ll}
\lambda & 0 \\
0       & \mu
\end{array} \right)
$$
\stackrel{\ast}{=}
$$
\left(
\begin{array}{ll}
\lambda & 0 \\
0       & \overline{\mu}
\end{array} \right);
$$
\:
$$
\left(
\begin{array}{ll}
\lambda & 1 \\
0       & \lambda
\end{array} \right)
$$
\not\stackrel{\ast}{=}
$$
\left(
\begin{array}{ll}
\lambda & 0 \\
0       & \overline{\lambda}
\end{array} \right).
$$
\qquad \qquad \qquad \qquad \qquad \qquad \qquad \qquad \qquad \qquad \qquad
\end{array}
$$

\end{exam}

\indent

The following theorem gives classification of linear maps from $\cc^2$ to $\cc^2$ up to topological conjugacy.

\begin{theorem}\label{klas-Ax-i-Cx}
Let $f$, $g: \cc^2 \rightarrow \cc^2$, $f(z)=Az$, $g(z)=Cz$ be linear maps.
$$
\begin{array}{ll}

\mbox{Then}\:\, f \stackrel{t}{\sim} g \:\: \mbox{if and only if} \:\:
$$
\left\{
\begin{array}{ll}
\rank(A_+)=\rank(C_+), \\
\rank(A_-)=\rank(C_-), \\
A_\infty = C_\infty, \\
A_0 \stackrel{\ast}{=} C_0.
\end{array} \right.
$$
\qquad &\qquad \qquad \qquad \qquad \qquad \qquad
\end{array}
$$
\end{theorem}

\begin{proof}
The maps $f(z)=Az$ and $g(z)=Cz$ from $\cc^2$ to $\cc^2$ are topologically conjugate if and only if the
corresponding linear maps $f_{\rr^4}(x)=A_{\rr^4}x$ and $g_{\rr^4}(x)=C_{\rr^4}x$ from $\rr^4$ to $\rr^4$ (with the real entried matrices) are topologically conjugate.
$$
\begin{array}{llllll}

\mbox{The map} \:\,f(z)=Az=
$$
\left(
\begin{array}{cc}
a & b \\
c & d
\end{array}\right)
$$
$$
\left(
\begin{array}{cc}
z_1 \\
z_2 \\
\end{array}\right),
$$
\:\, \mbox{from} \:\, \cc^2 \:\,\mbox{to}\:\, \cc^2, \:\, \mbox{corresponds to the map}
&\qquad &\qquad &\qquad &\qquad &\qquad
\end{array}
$$
$$
\begin{array}{lllll}

f_{\rr^4}(x)=A_{\rr^4}x=
$$
\left(
\begin{array}{ccccc}
a_1 & -a_2 & b_1 & -b_2  \\
a_2 &  a_1 & b_2 &  b_1  \\
c_1 & -c_2 & d_1 & -d_2  \\
c_2 &  c_1 & d_2 &  d_1
\end{array}\right)
$$
$$
\left(
\begin{array}{ccccc}
x_1 \\
x_2 \\
x_3 \\
x_4 \\
\end{array}\right),
$$
\:\, \mbox{from} \:\, \rr^4 \:\,\mbox{to}\:\, \rr^4.
\qquad & \qquad & \qquad & \qquad & \qquad
\end{array}
$$

The matrix $A_{\rr^4}$ has the following properties.

1) The eigenvalues of the matrix $A_{\rr^4}$ are:

$\lambda^{(1)}_{A_{\rr^4}}=\frac{1}{2}\left( a+d+\sqrt{(a-d)^2+4bc} \right)$;

$\lambda^{(2)}_{A_{\rr^4}}=\frac{1}{2}\left( a+d-\sqrt{(a-d)^2+4bc} \right)$;

$\lambda^{(3)}_{A_{\rr^4}}=\frac{1}{2}\left( \overline{a}+\overline{d}+\sqrt{(\overline{a}-\overline{d})^2+4\overline{b}\overline{c}} \right)$;

$\lambda^{(4)}_{A_{\rr^4}}=\frac{1}{2}\left( \overline{a}+\overline{d}-\sqrt{(\overline{a}-\overline{d})^2+4\overline{b}\overline{c}} \right)$.

2)    $\overline{\lambda^{(1)}_{A_{\rr^4}}}=\lambda^{(3)}_{A_{\rr^4}}$;

\quad $\overline{\lambda^{(2)}_{A_{\rr^4}}}=\lambda^{(4)}_{A_{\rr^4}}$.

3)    $\lambda^{(1)}_{A_{\rr^4}}=\lambda^{(1)}_A$;

\quad $\lambda^{(2)}_{A_{\rr^4}}=\lambda^{(2)}_A$.

4) If we denote: $\lambda^{(1)}_{A_{\rr^4}}=\alpha+i\beta$, $\alpha$, $\beta \in \rr$;

\qquad \qquad \qquad \quad \, $\lambda^{(2)}_{A_{\rr^4}}=\xi+i\mu$, $\xi$, $\mu \in \rr$,

then by 2-d property ${\cal R}_{A_{\rr^4}}$ is:
$$
\begin{array}{ll}

$$
\left(
\begin{array}{ccccc}
\alpha & -\beta  & 0    & 0   \\
\beta  &  \alpha & 0    & 0   \\
0      &  0      & \xi  & -\mu\\
0      &  0      & \mu  &  \xi
\end{array}\right),
$$
\quad \mbox{or}\:
$$
\left(
\begin{array}{ccccc}
\alpha & -\beta  &    1   &    0    \\
\beta  &  \alpha &    0   &    1    \\
0      &    0    & \alpha & -\beta  \\
0      &    0    & \beta  &  \alpha
\end{array}\right).
$$
\end{array}
$$

5) If the matrix ${A_{\rr^4}}$ has nonzero eigenvalues, whose moduli are less than 1, then the size of the matrix ${A_{\rr^4}}_+$ must be either $2\times 2$ or $4\times 4$, and $\sign\,(\,\det(  {A_{\rr^4}}_+ )\,)>0$ (it is easy to see by 4-th property).

Similarly it can be shown that $\sign\,(\,\det(  {A_{\rr^4}}_-  )\,)>0$.

Hence, by Proposition~\ref{klas-vsix-liniu-vidob-5-5} the maps $f_{\rr^4}(x)=A_{\rr^4}x$ and $g_{\rr^4}(x)=C_{\rr^4}x$ are topologically conjugate if and only if
\begin{equation} \label{lin-matr}
\left\{
\begin{array}{lllll}
\rank({A_{\rr^4}}_+)=\rank({C_{\rr^4}}_+), \\
\rank({A_{\rr^4}}_-)=\rank({C_{\rr^4}}_-), \\
{A_{\rr^4}}_\infty = {C_{\rr^4}}_\infty, \\
{A_{\rr^4}}_0 = {C_{\rr^4}}_0.
\end{array} \right.
\end{equation}

6) If the matrix ${A_{\rr^4}}$ has nonzero eigenvalues, whose moduli are less than 1, then by 3-d and 4-th properties:  $\rank({A_{\rr^4}}_+)=2\rank(A_+)$.

Similarly it can be shown that $\rank({A_{\rr^4}}_-)=2\rank(A_-)$.

7) If each of the matrices $A_{\rr^4}$ and $C_{\rr^4}$ has zero eigenvalues, then by 3-d property it is easy to see that ${A_{\rr^4}}_\infty = {C_{\rr^4}}_\infty$ if and only if $A_\infty = C_\infty$.

8) If each of the matrices $A_{\rr^4}$ and $C_{\rr^4}$ has eigenvalues, whose moduli are equal to 1, then by 3-d property and by the similar arguments as in the proof of Lemma~\ref{klas-liniu-vidob-C} (namely, the 5-th property) it can be shown that ${A_{\rr^4}}_0 = {C_{\rr^4}}_0$ if and only if $A_0 \stackrel{\ast}{=} C_0$.

By properties $6 - 8$ we rewrite the conditions~\eqref{lin-matr} as follows:

$$
\left\{
\begin{array}{ll}
\rank(A_+)=\rank(C_+), \\
\rank(A_-)=\rank(C_-), \\
A_\infty = C_\infty, \\
A_0 \stackrel{\ast}{=} C_0.
\end{array} \right.
$$

Consequently, maps $f$, $g: \cc^2 \rightarrow \cc^2$, $f(z)=Az$ and $g(z)=Cz$ are topologically
$$
\begin{array}{llll}

\mbox{conjugate if and only if}\:\:
$$
\left\{
\begin{array}{ll}
\rank(A_+)=\rank(C_+), \\
\rank(A_-)=\rank(C_-), \\
A_\infty = C_\infty, \\
A_0 \stackrel{\ast}{=} C_0.
\end{array} \right.
$$
\qquad & \qquad & \qquad & \qquad \qquad \qquad \qquad
\end{array}
$$
\end{proof}

\subsubsection{Classification of affine maps from $\cc^2$ to $\cc^2$.}

The ideas that we used for classification of affine maps from $\rr^2$ to $\rr^2$ up to topological conjugacy, can be applied for the same classification of affine maps from $\cc^2$ to $\cc^2$.

The following results have the similar proofs as in the real entried case.

\begin{theorem}\label{pro-neryx-tochk-C-2}
Let $f$, $g: \cc^2 \rightarrow \cc^2$, $f(z)=Az+b$ and $g(z)=Az$.

Then $f\stackrel{t}{\sim}g$ if and only if there is $q \in \cc^2$ such that $f(q)=q$.
\end{theorem}

\begin{remk}
Theorem~\ref{pro-neryx-tochk-C-2} is true for the corresponding maps from $\cc^n$ to $\cc^n$, $n\geq 1$.
\end{remk}

\begin{theorem}\label{osnovna-C-2}
Let $f$, $g: \cc^2 \rightarrow \cc^2$, $f(z)=Az+b$, $g(z)=Cz+d$ be affine maps.

If each of $f$ and $g$ has at least one fixed point, then $f\stackrel{t}{\sim}g$ if and only if
$$
\left\{
\begin{array}{lllll}
\rank(A_+)=\rank(C_+), \\
\rank(A_-)=\rank(C_-), \\
A_\infty = C_\infty, \\
A_0 \stackrel{\ast}{=} C_0.
\end{array} \right.
$$

If $f$ and $g$ have no fixed points, then\, $f\stackrel{t}{\sim}g$\, if and only if\, $\det A$ and $\det C$ are either simultaneously equal to 0 or simultaneously different from 0.
\end{theorem}

\section{Acknowledgements.} The author wishes to express his sincere gratitude to V. Sharko, S. Maksymenko and the others scholars of the department of topology of Institute of mathematics for many interesting discussions during the preparation of this paper.

\indent
Kyiv National Taras Shevchenko University,\\
64, Volodymyrska st.,
01033 Kyiv, Ukraine.\\
E-mail: Budnitska\underline{ }T@ukr.net

\end{document}